\documentclass[preprint,12pt,3p]{elsarticle}

\usepackage{amssymb}

\journal{the arXiv}

\usepackage{booktabs}
\usepackage{blindtext}
\usepackage{graphicx}
\usepackage{amsmath}
\usepackage{algorithm}
\usepackage[noend]{algpseudocode}

\usepackage{graphicx}
\usepackage{xcolor}
\usepackage{geometry} 
\usepackage{array} 
\usepackage{textcomp} 
\usepackage[all,cmtip]{xy} 
\usepackage{varioref}
\usepackage{caption}
\usepackage{subcaption}
\usepackage{epstopdf}

\usepackage{amssymb,amsthm,amsmath,amsfonts,latexsym,tikz}
\newtheorem{thm}{Theorem}[section]

\newtheorem{cor}[thm]{Corollary}
\newtheorem{lem}[thm]{Lemma}
\newtheorem{conj}[thm]{Conjecture}

\newtheorem{question}[thm]{Question}

\DeclareMathOperator{\girth}{girth}

\makeatletter
\def\BState{\State\hskip-\ALG@thistlm}
\makeatother

\def \N {\mathcal{N}}

\begin{document}

\begin{frontmatter}

\title{Distance Preserving Graphs}
\tnotetext[label0]{\em{This paper has been accepted in 45th Southeastern International Conference on
Combinatorics, Graph Theory, and Computing for which it has received the best student paper award.}}

\author{Emad Zahedi}
\address{Department of Mathematics, Michigan State University, \\
              East Lansing, MI 48824-1027, U.S.A.}
\address{Department of Computer Science and Engineering,	Michigan State University\\
							East Lansing, MI 48824, U.S.A.}


\ead{Zahediem@msu.edu}

\begin{abstract}
Given a graph $G$ then a subgraph $H$ is {\em isometric} if, for every pair of vertices $u,v$ of $H$, we have $d_H(u,v) = d_G(u,v)$. We say a graph $G$ is {\em distance preserving (dp)} if it has an isometric subgraph of every possible order up to the order of $G$.
We  consider how to add a vertex to a dp graph so that the result is a dp graph. This condition implies that chordal graphs are dp. We also find a condition on the $\girth$ of $G$ which implies that it is not dp. In closing, we discuss other work and open problems concerning dp graphs.
\end{abstract}


\end{frontmatter}



\section{Introduction}\label{Introduction}

Computing distance properties of large graphs such as real-world social networks which consist of  millions of nodes is extremely expensive.  
Recomputing distances in subgraphs of the original network will be even more costly.
A solution to remedy this issue would be to find subgraphs which have the same distances as the original network.
Such a subgraph is called {\em isometric}. Distance properties where isometric subgraphs come into play have been used in network clustering~\cite{NF2, nussbaum2013clustering}.

One family of graphs which has been studied in the literature involving isometric subgraphs is the set of distance-hereditary graphs. A {\em distance-hereditary} graph is a connected graph in which every connected induced subgraph of $G$ is isometric~\cite{howorka1977characterization}.
Distance-hereditary graphs have been studied in various papers~\cite{bandelt1986distance, damiand2001simple, hammer1990completely} since they were first described by Howorka~\cite{howorka1977characterization}. In this article, we relax this property by using a notion we call distance preserving. 

A graph is {\em distance preserving}, for which we use the abbreviation dp,  if it has an isometric subgraph of each possible order. The definition of a distance-preserving graph is similar to the one for distance-hereditary graphs, but is less restrictive. Because of this, distance-preserving graphs can have a more complex structure than distance-hereditary ones.

It is easy to see that trees are dp by removing leaves. In the present work we investigate conditions under which adding a vertex to a dp graph preserves the dp property. By applying this construction recursively to $K_1$, one can construct various families of dp graphs. We use this method to prove that chordal graphs (which include trees) are dp. As opposed to the acyclic case, the presence of certain cycles can cause a graph not to be dp. We show that if $G$ is a graph with $\girth (G) \geq 5$ and every vertex is either a cut vertex or in a cycle, then $G$ does not have any isometric subgraph of order $|V(G)|-1$ and so is not dp.

\section{Background}\label{Background}

In this paper every graph is finite, undirected, simple, and connected. In a graph $G=(V,E)$, a {\em path} is a sequence of distinct vertices $v_0, \dots, v_k$ such that $v_iv_{i+1}\in E(G)$ for $i=0,\dots, k-1$.  The length of the path is $k$, the number of edges. A {\em cycle} of a graph is a sequence of vertices $v_0, \dots,v_k$ which are distinct, except for $v_0=v_k$, and $v_iv_j\in E(G)$ if $|i-j|=1$  (mod $k$). The length of a cycle C is its number of edges. The {\em girth} of a graph $G$ is the smallest length of a cycle in $G$ and denoted by $\girth (G)$.

The {\em distance} between vertices $u,v$ in $G,\ d_G(u,v),$ is the minimal length of a path connecting these vertices.
 If it is clear from  context, we will use $d(u,v)$, instead of $d_G(u,v)$. A path $P$ from $u$ to $v$ with length $d_G(u,v)$ is called $u$-$v$ {\em geodesic} path. 
An induced subgraph $H$ of a graph $G$ is called an {\em isometric} subgraph if $d_H(a,b)=d_G(a,b)$, for every pair of vertices $a,b \in V(H)$, denoted by $ H \leq G $. 
A connected graph $G$ is called {\em distance preserving (dp)} if and only if it has an $i$-vertex isometric subgraph for every $1 \leq i \leq |V(G)|$.

If $ G $ is a graph and $ A \subseteq V(G) $ then $ G[A] $ denotes the subgraph induced by $ A $. The set of vertices adjacent to $v\in V$ is called its {\em neighborhood} and denoted $\N(v)$.  A vertex $v\in V(G)$ is called a {\em simplicial} vertex if $G[\N(v)]$ is a clique.
A graph $G$ is said to have a {\em simplicial elimination ordering} if there is an ordering $V(G) = \{v_1, \dots, v_{|V(G)|} \}$ such that $v_j$ is simplicial in $G[v_1,\cdots, v_{j}]$ for $1 \leq j  \leq |V(G)|$.

\section{Chordal Graphs}\label{chordal}
A {\em chordal} graph is a graph in which any cycle of length four or more has a chord. A graph $G$ is chordal if and only if $G$ has a simplicial elimination order~\cite{fulkerson1965incidence}.   The following lemma will permit us to prove that all chordal graphs are dp.

\begin{lem}\label{B} Let $v$ be a simplicial vertex in $G$. If $G-v$ is a dp graph then $G$ is dp.
\end{lem}

\begin{proof} Let $G'=G-v$ and $n=|V|$.  We claim it suffices to show that $G'$ is an isometric subgraph of $G$.  Indeed,  $G'$ will be an isometric subgraph of $G$ of order $n-1$.  And for $k<n-1$, the fact that $G'$ is dp implies that there is an isometric subgraph $H$ of $G'$.  But then $H$ is also isometric in $G$ because being an isometric subgraph is a transitive relation.

To show $G'$ is isometric in $G$, consider $x,y\in V(G')$.
Since $v$ is simplicial in $G$, $G[\N_{G}(v)]$ is an induced complete subgraph of $G$.  We claim that  in $G$, any $x$-$y$ geodesic  can not contain $v$. 
This implies that $d_G(x,y)=d_{G'}(x,y)$, and thus $G'$ will be isometric as desired.
Suppose, towards a contradiction, that  $P: x,\dots, u,v,w, \dots, y$ is an $x$-$y$ geodesic in $G$ 
then $u, w$ lie in $\N_{G}(v)$. But $G[\N_{G}(v)]$ is complete, so $uw\in E(G)$ and $\hat{P}:$ $x,u_1,\dots, u,w, \dots, y$ is another path from $x$ to $y$ in $G$ which is shorter than $P$. This contradicts the fact that $P$ is an $x$-$y$ geodesic and finishes the proof.
\end{proof}

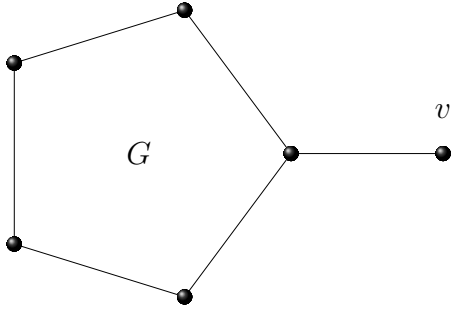
\begin{figure}
\centering
\begin{center}
    \begin{tikzpicture}

\draw[line width=.3pt] (4,0) to  (2,0);
\draw[line width=.3pt] (2,0) to  (.6,-1.9);
\draw[line width=.3pt] (2,0) to  (.6,1.9);
\draw[line width=.3pt] (.6,1.9) to  (-1.64,1.2);
\draw[line width=.3pt] (.6,-1.9) to  (-1.64,-1.2);
\draw[line width=.3pt] (-1.64,1.2) to  (-1.64,-1.2);

\node [below] at (0,0.3) {$G$};
\node [below] at (4,.8) {$v$};

\shade[shading=ball, ball color=black] (4,0) circle (.1);
\shade[shading=ball, ball color=black] (2,0) circle (.1);
\shade[shading=ball, ball color=black] (-1.64,1.2) circle (.1);
\shade[shading=ball, ball color=black] (-1.64,-1.2) circle (.1);
\shade[shading=ball, ball color=black] (.6,-1.9) circle (.1);
\shade[shading=ball, ball color=black] (.6,1.9) circle (.1);

	\end{tikzpicture}
\end{center}

\caption[]{A counterexample graph $G$ to the converse of Lemma~\ref{B}.}
\label{fig:figure1}
\end{figure}

The converse of the Lemma~\ref{B} is not true. In Figure~\ref{fig:figure1}, it is easy to check that $G$ is a dp graph and $G[\N_{G}(v)]$ is complete. But $G-v = C_5$ is not dp since removal of any vertex of the cycle results in a subgraph which is a path and not isometric in $C_5$.

The next corollary generalizes the fact mentioned previously that all trees are dp.

\begin{cor}\label{D} Chordal graphs are dp.
 \end{cor}

\begin{proof} Our proof is  by induction on $n=|V|$. The result  is clear when $n=1$.
Given a chordal graph $G$, let $v_1,v_2,...,v_n$ be a simplicial elimination order for its vertices. By induction $G'=G[v_1,v_2,..,v_{n-1}]$ is dp and $v_n$ is simplicial in $G$. Using Lemma~\ref{B} implies that $G$ is dp.
\end{proof}

We can relax the condition in Lemma~\ref{B} as follows.

\begin{thm}\label{E}
 Let $G$ contain a vertex $v$ such that every pair of non-adjacent $u, w\in \N(v)$ are in a $4$-cycle in $G$. If $G-v$ is dp then so is $G$.
\end{thm} 

\begin{proof}
As in the proof of Lemma~\ref{B}, it suffices to show that $G'=G-v$ is an isometric subgraph of $G$.  So take $x,y\in V(G')$ and an $x$-$y$ geodesic $P$  in $G$.  It suffices to show that there is an $x$-$y$ geodesic in $G'$ of the same length.  If  $P$ does not contain $v$, then it is a geodesic in $G'$ and we are done.  If $P$  contains $v$, say $P: x,\dots, u,v,w, \dots, y$.  If $uw$ is an edge of $G$ then we derive a contradiction as in the proof of Lemma~\ref{B}.  If $u$ and $w$ are not adjacent then, by the cycle hypothesis, there must be a vertex $z\neq v$ with $uz,zw\in E(G)$.  So, by the choice of $z$, the path 
$\hat{P}:  x,\dots, u,z,w, \dots, y$ is an $x$-$y$ geodesic in $G'$ with the same length as $P$.
\end{proof}


\section{Girth}\label{section:3}

As we have seen, connected acyclic graphs (namely, trees) are dp. On the other hand, as also previously mentioned, the $5$-cycle $C_5$ is not dp.  We now give a condition on the girth of a graph $G$ which implies that it is not dp.   Note the contrast with the cycle condition in Theorem~\ref{E}.

\begin{thm}\label{A} 
Let $G$ be a graph such that $\girth (G) \geq 5$ and such that  every vertex is either a cut vertex or in a cycle. Then $G$ is not dp.
\end{thm} 
\begin{proof}
 Assume that $G$ is dp so we can delete a vertex, say $ v$, to obtain an isometric subgraph of order $ |V(G)|-1 $ in $G$. Now $ v $  can not be a cut vertex since a disconnected subgraph of $G$ can not be isometric. Therefore $ v $ belongs to a cycle $C$ and there exist two vertices $u,w$ in $C$ such that $uvw$ is a path in $G$.
By assumption $\girth(G)\geq 5$ and so $G$ does not contain a $3$-cycle.  Thus $uw\notin E(G)$ and $d_G(u,w)=2$.
Since $G - v$ is isometric, $d_{G-v}(u,w)=2$ and consequently there is a vertex $\hat{v} \in V(G-v)$ so that $u, \hat{v}, w$ is a path in $G-v$. This implies $u,v,w,\hat{v},u$ is a 4-cycle in $G$ which contradicts  $\girth(G)\geq 5$.
Since the vertex $ v $ was arbitrary, $G$ has no isometric subgraph of order $|V(G)|-1$ and so is not dp.
\end{proof}

\section{Further Results and Future Work}\label{section:4}

Here is a list of other results which have been proved for dp graphs.
\begin{itemize}
\item We have found a necessary and sufficient condition for a given graph $G$ with a cut vertex to be dp. 
\item  We have shown that if $G$ is dp then so is the lexicographic product of $G$ with $H$ for any graph $H$.
\item  Call a graph {\em sequentially dp} if there is an ordering of the vertices $v_1, \dots, v_n$ such that  $G - \{v_ 1, \dots, v_s \} \leq G$ for $1 \leq s \leq |V (G)|$. We have proved that the Cartesian product of a sequentially dp graph $G$ with $H$ is dp for every dp graph $H$.
\end{itemize}

We end with some interesting questions and conjectures about distance preserving graphs.

\begin{conj}
If  $G$ be an $n$-vertex graph with minimum degree  $\delta(G) > n/2$ then $G$ is
dp. 
\end{conj} 

Nussbaum  and Esfahanian have shown that $\delta(G)\ge \frac{2n}{3}-1$  forces $G$ to be dp~\cite{NF2}. 

 \begin{conj} If $G$ does not contain an induced cycle of length $5$ or greater, then $G$ is a dp  graph.
\end{conj}

This conjecture is inspired by the way cycle lengths enter into Theorems~\ref{E} and~\ref{A}.

\begin{question}
 Prove or disprove that almost all graphs are dp. 
\end{question}

We note that almost all graphs have diameter two.  So the distance structure of such graphs may be simple enough to permit a proof of this last statement.

\bibliography{references}
\bibliographystyle{abbrv}



\end{document}